\documentclass[a4paper, 11pt, english]{article}
\usepackage[utf8]{inputenc}
\usepackage[T1]{fontenc}

\title{On the dimension of observable sets for the heat equation}
\author{A. Walton Green\footnote{Washington University, St. Louis, MO. {\tt awgreen@wustl.edu}; A.~W.~G. is supported in part by NSF grant DMS-2202813}, Kévin Le Balc'h\footnote{Inria, Sorbonne Université, Université de Paris, CNRS, Laboratoire Jacques-Louis Lions, Paris, France. {\tt kevin.le-balc-h@inria.fr}},
Jérémy Martin\footnote{Inria, Sorbonne Université, Université de Paris, CNRS, Laboratoire Jacques-Louis Lions, Paris, France. {\tt jeremy.a.martin@inria.fr.}
}, Marcu-Antone Orsoni\footnote{Université Laval, Québec, QC, Canada. {\tt marcu-antone.orsoni@mat.ulaval.ca} The paper was mainly written while M.-A. O. was a postdoctoral fellow at the University of Toronto Mississauga and the Centre for Nonlinear Analysis and Modeling.}}

\date{\today}

\usepackage[top=3cm, bottom=2cm, left=3cm, right=3cm]{geometry}	

\usepackage{enumitem}
\usepackage{array}

\usepackage{amsfonts}
\usepackage{amsmath}
\usepackage{amsthm}
\usepackage{amssymb}
\usepackage{mathtools}

\usepackage{multicol}
\usepackage{color}

\usepackage{bbm}

\usepackage{stmaryrd}

\usepackage[scr]{rsfso}

\usepackage{comment}

\usepackage[dvipsnames]{xcolor}

\usepackage{hyperref}

\usepackage{cleveref}

\Crefname{paragraph}{Section}{Sections}

\hypersetup{	
  colorlinks=true,
  breaklinks=true,
  urlcolor= red,
  linkcolor= red,
  citecolor=ForestGreen
}

\numberwithin{equation}{section}

\newtheorem{theorem}{Theorem}[section]
\newtheorem{definition}[theorem]{Definition}
\newtheorem{lemma}[theorem]{Lemma}
\newtheorem{proposition}[theorem]{Proposition}

\newtheorem{corollary}[theorem]{Corollary}
\newtheorem{remark}[theorem]{Remark}
\newtheorem{question}[theorem]{Question}
\numberwithin{equation}{section}

\newcommand{\norme}[1]{\left\lVert#1\right\rVert}
\newcommand{\norm}[1]{\left\lVert#1\right\rVert}
\newcommand{\abs}[1]{\left\vert#1\right\vert}
\DeclareMathOperator{\cp}{cap}

\newcommand{\ensemblenombre}[1]{\mathbb{#1}}
\newcommand{\N}{\ensemblenombre{N}}

\newcommand{\R}{} % Probleme LateXML
\renewcommand{\R}{\ensemblenombre{R}}

\newcommand{\nn}{\mathbb{N}}

\newcommand\bna{\begin{eqnarray*}}
\newcommand\ena{\end{eqnarray*}}

\newcommand\bnan{\begin{eqnarray}}%numÈrotÈe
\newcommand\enan{\end{eqnarray}}

\begin{document}

\maketitle

\begin{abstract}
We consider the heat equation on a bounded $C^1$ domain in $\mathbb{R}^n$ with Dirichlet boundary conditions. The primary aim of this paper is to prove that the heat equation is observable from any measurable set with a Hausdorff dimension strictly greater than $n - 1$. The proof relies on a novel spectral estimate for linear combinations of Laplace eigenfunctions, achieved through the propagation of smallness for solutions to Cauchy-Riemann systems as established by Malinnikova, and uses the Lebeau-Robbiano method. While this observability result is sharp regarding the Hausdorff dimension scale, our secondary goal is to construct families of sets with dimensions less than $n - 1$ from which the heat equation is still observable.
\end{abstract}

 \small
\tableofcontents
\normalsize

\section{Introduction}

\subsection{Observability and null-controllability of the heat equation}

Given \( T > 0 \), \(\Omega\) a bounded, open, connected subset of \(\mathbb{R}^n\), and \(\omega \subset \Omega\) a measurable set, this paper aims to establish observability estimates for the heat equation:
\begin{equation} \label{eq:Heat_equation}
    \begin{cases}
        \partial_t u - \Delta u = 0 & \text{in } (0,T) \times \Omega, \\ 
        u = 0 & \text{on } (0,T) \times \partial \Omega, \\
        u(0,\cdot) = u_0 & \text{in } \Omega.
    \end{cases}
\end{equation}

First, let us review some basic well-posedness properties of the heat equation \eqref{eq:Heat_equation}. According to \cite[Sections 7.1]{Eva10}, for every \( u_0 \in L^2(\Omega) \), there exists a unique weak solution \( u \in L^2(0,T; H_0^1(\Omega)) \cap H^1(0,T; H^{-1}(\Omega)) \) to the initial/boundary-value problem \eqref{eq:Heat_equation}. Using Sobolev embedding for mixed spaces, we deduce that \( u \in C([0,T]; L^2(\Omega)) \). Furthermore, due to the \( L^2-L^{\infty} \) regularizing effect of the heat equation (see, for instance, \cite[Proposition 3.5.7]{CH98}), we have \( u(T, \cdot) \in L^{\infty}(\Omega) \).

With this in hand, we introduce the following notion of observability for the heat equation \eqref{eq:Heat_equation} from \(\omega\) at time \( T \).

\begin{definition}
\label{def:ObsL2L2Linfty}
Let \( T > 0 \) and \(\omega \subset \Omega\) be a measurable set. The heat equation \eqref{eq:Heat_equation} is said to be observable from \(\omega\) at time \( T \) if there exists a positive constant \( C = C(\Omega, \omega, T) > 0 \) such that for every $u_0 \in L^2(\Omega)$, the solution $u$ of \eqref{eq:Heat_equation} satisfies
\begin{equation}
\label{eq:Observability_definition}
\|u (T,\cdot)\|_{L^{\infty}(\Omega)} \leq C \int_0^T \sup_{x \in \omega} |u(t,x)| \, dt.
\end{equation}
\end{definition}

\begin{remark}
The left-hand side of \eqref{eq:Observability_definition} is well-defined and finite due to the regularizing properties ensuring \( u(T, \cdot) \in L^{\infty}(\Omega) \). On the other hand, the right-hand side of \eqref{eq:Observability_definition} could potentially be infinite without additional assumptions on \(\Omega\), \(\omega\), or \( u_0 \). Additionally, at each time \( t \in (0,T) \), the term \( \sup_{x \in \omega} |u(t,x)| \) may not coincide with the standard \( L^{\infty} \)-norm, i.e., \( \|u(t,\cdot)\|_{L^{\infty}(\omega)} \).

Note that this observability notion \eqref{eq:Observability_definition} differs from the more standard one, typically referred to as the \(L^2-L^2\) observability estimate,
\begin{equation}
\label{eq:Observability_definition_standard}
\|u (T,\cdot)\|_{L^2(\Omega)} \leq C \left( \int_0^T \int_{\omega} |u(t,x)|^2 \, dt \, dx \right)^{1/2}.
\end{equation}
Firstly, all terms in \eqref{eq:Observability_definition_standard} are well-defined (and finite) due to the properties of the heat equation recalled above. Secondly, because of the Hilbert structure of \( L^2(\Omega) \) and \( L^2((0,T) \times \omega) \), \eqref{eq:Observability_definition_standard} is well-suited for employing duality methods and proving null-controllability results for the controlled heat equation; see, for instance, \cite[Theorem 2.48]{Cor07}. Finally, \eqref{eq:Observability_definition_standard} is also well-suited for cases with open sets \(\omega\) or measurable sets \(\omega\) where \( |\omega| > 0 \), where \( |\cdot| \) denotes the \( n \)-dimensional Lebesgue measure in \( \mathbb{R}^n \). Here, however, since we focus on measurable subsets \(\omega \subset \Omega\) where possibly \( |\omega| = 0 \), we will work with the more precise observability notion \eqref{eq:Observability_definition} in \Cref{def:ObsL2L2Linfty}.
\end{remark}

\medskip

From now and in all the following, we assume that $\Omega$ satisfies the following hypothesis.
\begin{equation}
   \text{The set}\  \Omega\ \text{is a bounded Lipschitz open connected set of $\R^n$, locally star-shaped.} \label{eq:assOmega}
    \tag{H}
\end{equation}
See \cite[Definition 4]{AEWZ14} for a precise definition of locally star-shaped domains. Note that from \cite[Remark 6]{AEWZ14} such domains include $C^1$ bounded domains, Lipschitz polyhedrons in $\R^n$ with $n \geq 3$ and bounded convex domains.

\subsection{Main results}

Our goals in this paper are twofold. Our main theorem provides observability results for \eqref{eq:Heat_equation} when the Hausdorff dimension of the observable set $\omega$ is strictly greater than $n-1$ (i.e. has codimension strictly less than $1$). 

Let us recall that for $s \geq 0$, the $s$-Hausdorff content of a set $E\subset \mathbb{R}^n$ is 
\begin{equation}\label{e:Hcontent} \mathcal{C}_{\mathcal{H}}^s (E) = \inf \left\{ \sum_j r_j ^s\ ;\ E \subset \bigcup_j B(x_j, r_j)\right\},\end{equation}
and the Hausdorff dimension of $E$ is defined as 
$$ \text{dim}_\mathcal{H} (E) = \inf \{s \geq 0\ ;\    \mathcal{C}_{\mathcal{H}}^s(E) =0 \}, \quad \text{codim}_{\mathcal H}(E) = n-\text{dim}_{\mathcal H}(E).$$
\begin{remark}
In this paper, we use the notion of $s$-Hausdorff content recalled in \eqref{e:Hcontent} instead of the more classical one of $s$-Hausdorff measure, defined for a set $E \subset \R^n$, 
\begin{equation}
    \label{eq:definehausdorffmeasure}
    \mathcal M_{\mathcal H}^s(E) = \lim_{\delta \to 0} \inf \left\{ \sum_j \mathrm{diam}(A_j) ^s\ ;\ E \subset \bigcup_j A_j,\ \mathrm{diam}(A_j) \leq \delta \right\}.
\end{equation}
Of course, the $s$-Hausdorff content and $s$-Hausdorff measure can differ but from \cite[Lemma 4.6]{Mat99} they both have the same negligible sets so $\mathcal{C}_{\mathcal{H}}^s (E) > 0 \Leftrightarrow \mathcal M_{\mathcal H}^s(E) > 0$ for $E \subset \R^n$ whence they both define the same Hausdorff dimension.
\end{remark}

Our main result is as follows.

\begin{theorem}
\label{thm:observability_main_result}

Let $T>0$ and $\omega \subset \Omega$ be a measurable subset satisfying $\mathrm{dim}_{\mathcal H}(\omega) > n-1$. Then the heat equation \eqref{eq:Heat_equation} is observable from $\omega$ at time $T$.

More precisely, if $\delta >0$ and $\omega \subset \Omega$ is a measurable subset satisfying $\mathcal C^{n-1+\delta}_{\mathcal H}(\omega)>0$, then there exists $C=C(\Omega, \delta, \omega)>0$ such that for every $T>0$, for every $u_0 \in L^2(\Omega)$, the solution $u$ of \eqref{eq:Heat_equation} satisfies
\begin{equation}
\label{eq:Observability_quantitative}
\|u (T,\cdot)\|_{L^{\infty}(\Omega)} \leq C e^{C/T} \int_0^T \sup_{x \in \omega} |u(t,x)|dt.
\end{equation}

\end{theorem}

\begin{remark}
Note that the observability estimate \eqref{eq:Observability_quantitative} is quantitative in function of $T$ only. If one wants to obtain a quantitative estimate in function of the Hausdorff content of $\omega$, we can establish the following statement in the spirit of \cite{AEWZ14}. For $x_0 \in \Omega$, $R \in (0,1]$ be such that $B(x_0, 4R) \subset \Omega$, for every $m, \delta > 0$, there exists $C=C(\Omega,\delta, R, m)>0$ such that for every measurable subset $\omega \subset B(x_0,R)$ with $\mathcal C^{n-1+\delta}_{\mathcal H}(\omega) \geq m$, one has that for every $T>0$, for every $u_0 \in L^2(\Omega)$, the solution $u$ of \eqref{eq:Heat_equation} satisfies
\begin{equation}
\label{eq:Observability_quantitative_mT}
\|u (T,\cdot)\|_{L^{\infty}(\Omega)} \leq C e^{C/T} \int_0^T \sup_{x \in \omega} |u(t,x)|dt.
\end{equation}
\end{remark}

\bigskip

Some comments with respect to the geometrical assumptions on $\Omega$ and $\omega$ are in order.
\begin{itemize}
    \item  When $\omega \subset \Omega$ is a \textit{\textbf{nonempty open set}}, \cite{FR71} proved \eqref{eq:Observability_definition_standard} in the one-dimensional case by moment's method, then more than twenty years later, \cite{LR95} and \cite{FI96} generalized independently this result to the multi-dimensional case by using Carleman estimates. When $\Omega$ is $C^{\infty}$-smooth, \cite{LR95} actually used local elliptic Carleman estimates to prove a spectral estimate for linear combination of eigenfunctions (see \cite[Theorem 4.3]{JL99} for a precise statement), then the authors introduced the now so-called Lebeau-Robbiano's method for passing from a spectral estimate to the observability of the heat equation (see for instance \cite[Theorem 2.2]{Mil10} or \cite[Theorem 2.1]{BPS18}). When $\Omega$ is $C^2$, \cite{FI96} directly proved observability estimates for the heat equation with the help of global parabolic Carleman estimates. Our proof of \Cref{thm:observability_main_result} follows Lebeau-Robbiano's method, so our crucial points would consist in deriving new spectral estimates.
    \item When $\omega$ is only assumed to be measurable with \textit{\textbf{positive Lebesgue measure}} i.e. $|\omega|>0$ and $\Omega$ is Lipschitz and locally star-shaped, \cite{AEWZ14} proved \eqref{eq:Observability_definition_standard}. 
Their result is actually based on the obtaining of spectral estimate of Jerison-Lebeau's type together with the analytic properties of linear combination of eigenfunctions inside the domain $\Omega$. It is worth mentioning that as far as we know the obtaining of \eqref{eq:Observability_definition_standard} for general bounded Lipschitz domains $\Omega$ with observation localized in measurable sets $\omega$ such that $|\omega|>0$ or even for nonempty open sets $\omega$ is open (see \cite{Bar20} for some progress in that direction).
\item When $\Omega$ is $W^{2, \infty}$, \cite{BM23} proved \eqref{eq:Observability_definition} for measurable subsets $\omega$ with possibly zero-Lebesgue measure but with \textit{\textbf{positive Hausdorff content}} $\mathcal C^{n-1+\delta}_{\mathcal H}(\omega)>0$ for a $\delta \in (0,1)$ that is a priori close to $1$ by using recent quantitative propagation of smallness results for gradient of solutions to $A$-harmonic functions from \cite{LM18}. In particular, they went a little bit beyond \cite{AEWZ14} proving the observability from measurable sets that are of zero $n$-dimensional Lebesgue measure. Their result is also valid replacing the flat Laplacian $-\Delta$ by a divergence elliptic operator with Lipschitz coefficients, i.e. $-\mathrm{div}(A(x) \nabla \cdot)$ with a matrix-valued Lipschitz and uniformly elliptic diffusion $A$. Note also that it was the first observability result for the heat equation from a measurable set with zero-Lebesgue measure in the multi-dimensional case.
\item In the one-dimensional case, i.e. $n=1$, \cite{Y.Zhu24} exactly obtained \Cref{thm:observability_main_result} replacing the flat Laplacian $-\partial_{x}^2$ by a rough divergence elliptic operator,  $-\partial_{x}(a(x) \partial_x \cdot)$ where $a \in L^{\infty}(\Omega)$ and uniformly elliptic. Their main result is a \cite{LM18}-like quantitative propagation of smallness result for gradient of $a$-harmonic functions in the two-dimensional case for sets with Hausdorff dimensions $\mathrm{dim}_{\mathcal H}(E) >0$. It is obtained exploiting the very specific relation between $a$-harmonic functions and holomorphic functions in the complex plane. 
\item The assumption made in Theorem \ref{thm:observability_main_result} on $\omega$, that $\text{dim}_{\mathcal H}(\omega)>n-1$, is sharp in the following sense: for every $M>0$, there exists a subset $\omega$ with $\mathcal C^{n-1}_{\mathcal H}(\omega)\geq M$ such that \eqref{eq:Observability_definition} does not hold from $\omega$. Indeed, assume that $\Omega$ is $C^1$ and let $\omega := \{ x \in \Omega\ ;\ \varphi_{\lambda}(x) = 0\}$ be a nodal set of a Laplace eigenfunction $\varphi_{\lambda}$ associated to the eigenvalue $\lambda >0$ with Dirichlet boundary conditions on $\partial\Omega$ i.e.
\begin{equation}
	\label{eq:Laplaceeigenfunction}
		\left\{
			\begin{array}{ll}
				-\Delta \varphi_{\lambda} =\lambda \varphi_{\lambda} & \text{ in } \Omega,\\
				\varphi_{\lambda} = 0 & \text{ on }   \partial\Omega,
			\end{array}
		\right.
\end{equation}
such that $\|\varphi_{\lambda}\|_{L^{\infty}(\Omega)}=1$. Then it is a well-known fact that 
$\mathcal C^{n-1}_{\mathcal H}(\omega)\geq c \sqrt{\lambda}$ for $\lambda$ sufficiently large, see \cite{LMNN21}. Taking $\varphi_{\lambda}$ as a initial data, we obtain that the solution $u$ of \eqref{eq:Heat_equation} is given by $u(t,x) = e^{-\lambda t} \varphi_{\lambda}(x)$ so the left hand side of the observability estimate \eqref{eq:Observability_definition} is bounded from below by $e^{-\lambda T}$, while the right hand side vanishes. Hence, the inequality \eqref{eq:Observability_definition} cannot hold for this specific $\omega$.

\begin{remark}
\label{rmk:approximate_observability}
Actually, in the last point an even weaker property fails, the so-called approximate observability. The heat equation \eqref{eq:Heat_equation} is said to be approximately observable in time $T>0$ from $\omega$ if $u_{|(0, T)\times \omega}=0$ implies $u(T, \cdot)=0$. This is equivalent to $\omega$ not being included in the nodal set of any eigenfunction \eqref{eq:Laplaceeigenfunction} (see Appendix \ref{app:approximate_observability}). It is clear that observability \eqref{eq:Observability_definition} is stronger than approximate observability. To avoid that issue of having a $(n-1)$-dimensional subset $\omega$ contained in a nodal set, \cite{CZ04} allowed $\omega$ to vary in time. In a different way, that obstruction is removed in \cite{GL20} using both Dirichlet and Neumann traces for the observation. They proved that the $L^2$-observability from a nonempty interior hypersurface always holds in that setting. Since every singular set $\mathcal{S}=\{x\in \Omega: \varphi_\lambda(x)=|\nabla \varphi_\lambda|(x)=0\}$ has $(n-2)$-Hausdorff dimension (see e.g. \cite{Han94}) and in view of Theorem \ref{thm:observability_main_result}, it would be interesting to know if this latter result still holds for every subset $\omega$ with Hausdorff dimension strictly greater than $n-2$. 
\end{remark}

\end{itemize}

\bigskip

From \Cref{thm:observability_main_result} and a duality argument in the spirit of classical H.U.M. \cite[Theorem 2.48]{Cor07} we can deduce a null-controllability result for
the associated controlled heat equation
 \begin{equation}
	\label{eq:heat_bounded_controlled}
		\left\{
			\begin{array}{ll}
				  \partial_t y  -\Delta y= h 1_{\omega} & \text{ in }  (0,T) \times \Omega, \\
				  y = 0 & \text{ in }  (0,T) \times \partial\Omega,\\
				y(0, \cdot) = y_0 & \text{ in }\Omega.
			\end{array}
		\right.
\end{equation}
In the controlled system \eqref{eq:heat_bounded_controlled}, at time $t \in [0,+\infty)$, $y(t,\cdot) : \Omega \to \R$ is the state and $h(t,\cdot) : \omega \to \R$ is the control.

In the following, we denote by $\mathcal M(\Omega)$ the space of Borel measures on $\Omega$. Recall that the norm of $\mathcal{M}(\Omega)$ is defined as
\begin{equation}
    \label{eq:normeborel}
    \| \mu \|_{\mathcal M(\Omega)} = \sup_{f \in C_c^0(\Omega)} \frac{\left|\int_{\Omega} f d \mu \right|}{\|f\|_{\infty}}.
\end{equation}

\begin{theorem}
\label{thm:controlledbounded}
Let $T>0$ and $\omega \subset \Omega$ be a closed measurable subset satisfying $\text{dim}_{\mathcal H}(\omega) > n-1$. Then \eqref{eq:heat_bounded_controlled} is null-controllable at time $T$, that is for every $y_0 \in L^2(\Omega)$, there exists $h \in L^{\infty}(0,T;\mathcal M(\Omega))$ supported in $(0,T)\times\omega$ such that the (very)-weak solution $y$ of \eqref{eq:heat_bounded_controlled} satisfies $y(T,\cdot) = 0$. 

More precisely, if $\delta >0$ and $\omega \subset \Omega$ is a closed measurable set satisfying $\mathcal C^{n-1+\delta}_{\mathcal H}(\omega)>0$, there exists $C=C(\Omega, \delta, \omega)>0$, for every $y_0 \in L^2(\Omega)$, there exists $h \in L^{\infty}(0,T;\mathcal M(\Omega))$ supported in $(0,T)\times\omega$ satisfying
\begin{equation}
    \label{eq:estimatecontrolmanifoldswild}
    \sup_{t \in (0,T)} \|h(t)\|_{\mathcal M(\Omega)} \leq C  e^{\frac CT} \norme{y_0}_{L^2(\Omega)},
\end{equation}
such that the (very)-weak solution $y$ of \eqref{eq:heat_bounded_controlled} satisfies $y(T,\cdot) = 0$.
\end{theorem}
The proof of \Cref{thm:controlledbounded} will be omitted because it is a standard adaptation of \cite[Section 5]{BM23} starting from \Cref{thm:observability_main_result}. We insist on the fact that a first difficulty in the analysis is to give a meaning of the concept of (very)-weak solution to \eqref{eq:heat_bounded_controlled} with a control $h=h(t)$, taking its values in $\mathcal M(\Omega)$ and supported in $\omega$. \\
\medskip

Despite the sharpness in the scale of Hausdorff dimension of \Cref{thm:observability_main_result} just demonstrated, our second goal is to present specially constructed families of sets with dimension less than or equal to $n-1$ for which the heat equation remains observable; see \Cref{sec:MeasurablesetsLower} below. In fact, from \cite{Dol73}, we know that the heat equation \eqref{eq:Heat_equation} in $\Omega=(0,1)$ can be observed from a single, specially chosen point $\omega = \{x_0\} \subset \Omega$. Starting from this example in the one-dimensional case, we ask whether this type of specific set can be extended to the multi-dimensional case. It turns out that, when $\Omega = (0,1)^n$ with $n \geq 2$, it can be shown that \eqref{eq:Heat_equation} is not observable from a finite number of points; see \Cref{prop:MultiD_finite_observable_set} below. Based on this negative result, we employ a different strategy to construct sets with small dimension. First, in the one-dimensional case, when $\Omega=(0,1)$, using an elementary Remez-type inequality, see \Cref{prop:prop}, we exhibit a large class of sets from which the heat equation is observable; see \Cref{cor:one-dim-gamma}. In particular, this class includes all sets with positive logarithmic capacity (and therefore, by Frostman's lemma, all sets of positive dimension). Additionally, it contains many well-spaced countable sets (thus having zero logarithmic capacity), a prototypical example being $ \omega_\alpha = \{k^{-\alpha} : k = 1, 2, \ldots \}$ with $\alpha > 0$. As a result, by observing that one can take “Cartesian products of spectral estimates”, see \Cref{prop:spectral_cartesian_product2} below, we are able to construct closed observable sets of zero dimension in $\Omega = (0,1)^n$ by taking $\omega = \omega_\alpha \times \omega_\alpha \times \cdots \times \omega_\alpha$. We also refine the strategy by taking “Cartesian products of an observability inequality and a spectral estimate”; see \Cref{prop:Obversability+spectral} below. For instance, this allows us to show that $\{x_0\} \times \omega'$ where $\omega' \subset \Omega' \subset \R^{n-1}$ is an observation set for $(0,1) \times \Omega'$, assuming $x_0$ is a carefully chosen point and $\dim_{\mathcal H}(\omega') > n-2$. A few natural open questions arise immediately. First, is $\omega_{\text{exp}} = \{2^{-k} : k = 1, 2, \ldots \}$ an observation set in $(0,1)$? Secondly, in higher dimensions, could certain algebraic curves serve as replacements for a single point with prescribed algebraic properties? At present, it remains an open question whether simple examples such as a line with irrational slope or a parabola are observation sets for the heat equation on $(0,1)^2$; see \Cref{q:line} below.

\section{Proof of the main observability estimate, i.e. Theorem~\ref{thm:observability_main_result}}

In order to prove \Cref{thm:observability_main_result}, we will follow the classical Lebeau-Robbiano's method, \cite{LR95}, consisting in establishing first spectral estimate. The unbounded operator $-\Delta$ with domain $H^2(\Omega) \cap H_0^1(\Omega)$ is self-adjoint with compact resolvant in $L^2(\Omega)$ so it admits an orthonormal basis of eigenfunctions $(\varphi_k)_{k \geq 1}$, associated to the sequence of real eigenvalues $(\lambda_k)_{k \geq 1}$. Given $\Lambda >0$, we introduce the spectral projector $\Pi_{\Lambda}$ as follows
\begin{equation}\label{e:spec-proj}
    \Pi_{\Lambda} u = \sum_{\lambda_k \leq \Lambda} \langle u, \varphi_k\rangle_{L^2} \varphi_k,\qquad \forall u   \in L^2(\Omega).
\end{equation}

\subsection{The spectral estimate}

The main result of this section states the following spectral estimates.
\begin{theorem}
\label{thm:spectralopendirichlet}
For all $\delta >0$, for every $\omega \subset \Omega$ such that $\mathcal{C}_{\mathcal H}^{n-1+\delta}(\omega) >0$, there exists $C=C(\Omega,\delta,\omega)>0$ such that for every $\Lambda >0$, we have
\begin{equation}
\label{eq:spectraldirichlet}
   \norme{\Pi_{\Lambda} u}_{L^{2}(\Omega)} \leq C e^{C \sqrt{\Lambda}} \sup_{x \in \omega} \left|(\Pi_{\Lambda} u)(x) \right|\qquad \forall u \in L^2(\Omega).
\end{equation}
\end{theorem}

On the one hand, \Cref{thm:spectralopendirichlet} improves the spectral estimates from \cite{AEWZ14}, \cite{AE13}, \cite{BM23} with respect to the Hausdorff dimension of the observation set and extends \cite{Y.Zhu24} to the multi-dimensional setting. On the other hand, \Cref{thm:spectralopendirichlet} is only valid for an elliptic operator with real analytic metric while the spectral estimate of \cite{BM23} is valid for $W^{1, \infty}$ metric and the one of \cite{Y.Zhu24} is valid for $L^{\infty}$ metric. We focus our attention on the case of the pure Laplacian (constant coefficients), but it is clear that the method extends to the case of real-analytic coefficients. Our main tool for proving \Cref{thm:spectralopendirichlet} is a propagation of smallness result for gradients of harmonic functions, \Cref{lem:malinnikovaharmonic} below, which is derived from a result of Malinnikova \cite{Mal04} on generalized Cauchy-Riemann systems. Then, this result is applied to the spectral inequality using the well-known lifting strategy of Jerison and Lebeau \cite{JL99}. 
\begin{remark}
We emphasize that we only estimate the $L^2(\Omega)$-norm of linear combination of eigenfunctions in \eqref{thm:spectralopendirichlet}, instead of the $L^{\infty}(\Omega)$-norm considered for instance in \cite{BM23} and \cite{Y.Zhu24}. We can replace the $L^2$-norm by the $L^{\infty}$-norm but assuming more regularity assumption on the boundary of $\Omega$ a priori. For instance, for $\Omega$ a $C^{\infty}$-smooth domain, let us remark that we have $ -\Delta (\Pi_{\Lambda} u) = \sum_{\lambda_k \leq \Lambda} \lambda_k \langle u, \varphi_k \rangle \varphi_k\ \text{in}\ \Omega$ and $\Pi_{\Lambda} u = 0\ \text{on}\ \partial\Omega$
so by elliptic regularity in $H^m(\Omega)$ for $0 \leq m \leq n$, Parseval's identity, the spectral estimate \eqref{eq:spectraldirichlet} and Sobolev embeddings we have
\begin{equation}
\label{eq:LinftyL2LinearCombination}
    \|\Pi_{\Lambda} u\|_{L^{\infty}(\Omega)} \leq C e^{C \sqrt{\Lambda}} \norme{\Pi_{\Lambda} u}_{L^{2}(\Omega)} \leq  C e^{C \sqrt{\Lambda}} \sup_{x \in \omega} \left|(\Pi_{\Lambda} u)(x) \right|.
\end{equation}
\end{remark}

\subsection{Proof of the spectral estimate}

We begin by deriving the following propagation of smallness estimate from the main result of \cite{Mal04}.

\begin{lemma}
\label{lem:malinnikovaharmonic}
Let $d \geq 2$, $m,\delta >0$, $\mathcal K \subset \subset B_1 \subset \R^d$. There exists $C=C(d,\mathcal K,\delta,m)>0$ and $\alpha = \alpha(d,\mathcal K,\delta,m) \in (0,1)$ such that for any $\mathcal E \subset \subset B_1$ which is a subset of a hyperplane and satisfies $\mathcal{C}_{\mathcal H}^{d-2+\delta}(\mathcal E) \ge m$ there holds, for every $u : B_2 \to \R$ satisfying $\Delta u = 0$ in $B_2$,
\begin{equation}
\label{eq:gradientmalinnikova}
    \sup_{x \in \mathcal K} |\nabla u(x)| \leq C \sup_{x \in \mathcal E} |\nabla u(x)|^{\alpha} \|\nabla u\|_{L^{\infty}(B_{1})}^{1- \alpha}. 
\end{equation}
\end{lemma}

In other words, \Cref{lem:malinnikovaharmonic} tells us that we are able to propagate smallness for gradients of harmonic functions from sets of Hausdorff dimension $d-2 + \delta$ for every $\delta >0$. Note that \cite{LM18} obtains the same result for $A$-harmonic function with $A$ satisfying elliptic and Lipschitz assumptions and for measurable sets $\mathcal E$ satisfying $\mathcal{C}_{\mathcal H}^{d-1-\delta}(\mathcal E) >0$ for some $\delta>0$ a priori small. As we will see during the proof, \Cref{lem:malinnikovaharmonic} is a straightforward corollary of \cite[Main Result]{Mal04}. Finally, as stated in \cite[Section 3.2]{Mal09}, one can certainly replace the standard Laplace operator $\Delta$ in \Cref{lem:malinnikovaharmonic} by a second order uniformly elliptic operator $L$ with analytic coefficients and propagation of smallness of gradient of solutions to $Lu=0$ still hold. This will explain why \Cref{thm:spectralopendirichlet} still hold for $L = \mathrm{div}(A \nabla \cdot)$ with $A$, a matrix-valued uniformly elliptic with real-analytic entries.
\begin{proof}The tedious part of the proof involves verifying that
$$F = (F_1, \dots, F_d) = \nabla u = \left(\frac{\partial u}{\partial_{x_1}}  , \dots, \frac{\partial u}{\partial_{x_d}} \right),$$ 
is a solution to a generalized Cauchy-Riemann system which is a factor of the Laplacian. Once that is established, we immediately obtain from \cite[Main Result]{Mal04} that \eqref{eq:gradientmalinnikova} holds for $\mathcal E$ such that the capacity $\cp_{d-2+\epsilon}(E) \ge \tilde m$ for some $\epsilon,\tilde m >0$. It is not relevant for our purposes to define capacity, only to note the well-known quantitative relationship between capacity and Hausdorff content facilitated by the Frostman lemma \cite[Theorem 8.8]{Mat99}. In particular, since $\mathcal C^{d-2+\delta}_{\mathcal H}(\mathcal E) \ge m$, there exists $\tilde m = \tilde m(m,\delta)>0$ such that $\cp_{d-2+\delta/2}(\mathcal E) \ge \tilde m$. Thus the lemma is proved apart from verifying that $F$ satisfies the desired generalized Cauchy-Riemann system. To this end, notice first that
\begin{equation}
\label{eq:generalizedCGR}
    \sum_{j=1}^d \frac{\partial F_j}{\partial_{x_j}} = 0,\qquad \frac{\partial F_j}{\partial_{x_k}} = \frac{\partial F_k}{\partial_{x_j}},\ 1 \leq j < k \leq d.
\end{equation}
Indeed, the first equation of \eqref{eq:generalizedCGR} corresponds to $\mathrm{div}(\nabla u) = 0$ while the second corresponds to Schwarz's theorem on mixed derivatives by using the fact that because $u$ is harmonic on $B_2$ then it is $C^{\infty}$-smooth on $B_1$. By defining the operator $A$ as 
\begin{equation}
    \label{eq:defA}
    A : \R^d \to \R^m\qquad AF = \left(\sum_{j=1}^d \frac{\partial F_j}{\partial_{x_j}}, \left(\frac{\partial F_j}{\partial_{x_k}} - \frac{\partial F_k}{\partial_{x_j}}\right)_{ 1 \leq j < k \leq d}\right),
\end{equation}
with $m = 1+(d(d-1))/2$. The last equations of $A F = 0$ imply by Poincaré lemma that $F$ is locally a gradient. Then the first equation of $AF=0$ implies that $F$ is locally a gradient of a harmonic function so each component of $F$ is harmonic then the system of equations \eqref{eq:generalizedCGR} is a generalized Cauchy-Riemann system in the sense of \cite[Section 4.1]{Mal04}. Finally, the operator $A$ defined in \eqref{eq:defA} is a factor of the Laplacian in the sense of \cite[Definition 4.1]{Mal04} because
\begin{multline*}
    \frac{\partial}{\partial_{x_i}} \left( \sum_{j=1}^d \frac{\partial F_j}{\partial_{x_j}} \right) -  \sum_{k < i } \frac{\partial}{\partial_{x_k}} \left(\frac{\partial F_k}{\partial_{x_i}} - \frac{\partial F_i}{\partial_{x_k}}\right) +  \sum_{i < k \leq d} \frac{\partial}{\partial_{x_k}} \left(\frac{\partial F_i}{\partial_{x_k}} - \frac{\partial F_k}{\partial_{x_i}}\right)\\
    = \Delta F_i,\ 1 \leq i \leq d.
\end{multline*} 
\end{proof}

With \Cref{lem:malinnikovaharmonic} at hand, we can prove \Cref{thm:spectralopendirichlet}.
\begin{proof}[Proof of \Cref{thm:spectralopendirichlet}]
From the hypothesis that $\dim_{\mathcal H}(\omega)>n-1$, we get that there exists $\delta >0$ such that 
\begin{equation}
\label{eq:omegacontent}
  \mathcal C^{n-1+\delta}_{\mathcal H}(\omega) >0.
\end{equation}
We then split the proof into two steps.

\medskip
\textbf{Step 1: Observation on a suitable ball.} 

First, let us prove that there exist $x_0 \in \Omega$, $r>0$ such that 
\begin{equation}
    \label{eq:suitableball}
    m:=\mathcal C^{n-1+\delta}_{\mathcal H}(\omega \cap B(x_0,r)) >0\ \text{and}\ B(x_0,4r) \subset \Omega.
\end{equation}
Note that $\Omega = \cup_{k \in \N^*} \Omega_k$ where $\Omega_k = \{x \in \Omega\ ;\ \mathrm{dist}(x, \partial\Omega) \geq 1/k\}$. By sub-additivity of the Hausdorff-content and \eqref{eq:omegacontent} we have that there exists $k \in \N^*$ such that $\mathcal C^{n-1+\delta}_{\mathcal H}(\omega \cap \Omega_k) >0$. Then, there exists $x_0 \in \Omega_k$ such that $\mathcal C^{n-1+\delta}_{\mathcal H}(\omega \cap B(x_{0},1/8k)) >0$ so \eqref{eq:suitableball} holds with such an $x_0$ and $r=1/8k$ because $\mathrm{dist}(x_0,\partial\Omega) \geq 1/k$.\\

For $u \in L^2(\Omega)$ and $\Lambda \geq 1$, we only need to prove $\eqref{eq:spectraldirichlet}$ assuming that
\begin{equation}
    \label{eq:defu}
    u = \Pi_{\Lambda} u.
\end{equation}
Then we have from \cite[Theorem 3]{AEWZ14},
\begin{equation}
\label{eq:spectraldirichletRestrictedToballL2App}
   \norme{u}_{L^{2}(\Omega)} \leq C e^{C \sqrt{\Lambda}} \sup_{x \in B(x_0,r)} \left|u(x) \right|.
\end{equation}
Note that we have used \cite{AEWZ14} instead of the more classical reference \cite{JL99} because we are only assuming that $\Omega$ is a bounded Lipschitz domain, locally star-shaped.

\medskip

\textbf{Step 2: Observation on $\omega$.}
Let us take 
\begin{equation}
\label{eq:defhatu}
     \hat{u}(x,t) =  \sum_{\lambda_k \leq \Lambda} \left\langle u, \varphi_k \right\rangle_{L^2(\Omega)} \frac{\sinh(\sqrt{\lambda_k }t )}{\sqrt{\lambda_k}} \varphi_k(x)\qquad (x,t) \in \Omega \times (-2,+2).
\end{equation}
One can check that 
\begin{equation}
\label{eq:equationhatu}
    \Delta_{x,t} \hat{u} = 0 \qquad (x,t) \in \Omega \times (-2,+2).
\end{equation}
Up to a translation, let us assume that $x_0 = 0$ and reducing $r>0$ if necessary, let us consider $B_{1}$ the unit ball in $\R^{n+1} = \R^d$ and $\mathcal K \subset \R^{d}$ such that
$$ B(x_0,r) \times (-1/2,+1/2) \subset \subset \mathcal K \subset\subset B_1 \subset \subset \Omega \times (-1,+1),$$
and
$$ \mathcal E = (\omega \cap B(x_0,r)) \times \{0\} \subset  \R^{n} \times \{0\}.$$
Then we can apply the propagation of smallness for gradients of harmonic functions i.e. \eqref{eq:gradientmalinnikova} from \Cref{lem:malinnikovaharmonic}, 
\begin{equation}
    \sup_{(x,t) \in \mathcal K} |\nabla \hat{u}(x,t)| \leq C \sup_{(x,t) \in \mathcal E} |\nabla \hat{u}(x,t)|^{\alpha} \|\nabla \hat{u}\|_{L^{\infty}(B_1)}^{1- \alpha}. 
\end{equation}
In particular we get
\begin{equation}
\label{eq:PropagationSmallness_Proof}
    \sup_{x \in B(x_0,r)} |\partial_t \hat{u}(x,0)| \leq C \left(\sup_{x \in \omega \cap B(0,r)} |\partial_t \hat{u}(x,0)|\right)^\alpha \left\|\hat u\right\|^{1-\alpha}_{W_{t,x}^{1, \infty}(B_1)}.
\end{equation}
The left hand side of \eqref{eq:PropagationSmallness_Proof} exactly gives
\begin{equation}
\label{eq:lefthandside}
     \sup_{x \in B(x_0,r)} |\partial_t \hat{u}(x,0)| =  \sup_{x \in B(x_0,r)} |u(x)|.
\end{equation}
The first right hand side term of \eqref{eq:PropagationSmallness_Proof} exactly gives
\begin{equation}
\label{eq:righthandside1}
    \sup_{x \in \omega \cap B(x_0,r)} |\partial_t \hat{u}(x,0)| = \sup_{x \in \omega \cap B(x_0,r)} |u(x)|  \leq \sup_{x \in \omega} |u(x)|.
\end{equation}
By \eqref{eq:equationhatu}, we have that $\hat{u}$ is harmonic so is $C^{\infty}$-smooth by Weyl's lemma. Then by local estimates on harmonic functions, see for instance \cite[Section 2.2, Theorem 7]{Eva10} and Parseval's equality applied to \eqref{eq:defhatu}, the second right hand side term of \eqref{eq:PropagationSmallness_Proof} can be estimated as
\begin{equation}
\label{eq:righthandside2}
    \left\|\hat u\right\|_{W_{t,x}^{1, \infty}(B_1)} \leq C \left\|\hat u\right\|_{L^2(\Omega \times (-1,+1))}\leq C e^{C \sqrt{\lambda}} \norme{u}_{L^{2}(\Omega)}.
\end{equation}
Therefore, we have from \eqref{eq:PropagationSmallness_Proof}, \eqref{eq:lefthandside}, \eqref{eq:righthandside1} and \eqref{eq:righthandside2},
\begin{equation}
\label{eq:FinStep2}
    \sup_{x \in B(x_0,r)} |u(x)| \leq C e^{C \sqrt{\lambda}} (\sup_{x \in \omega} |u(x)|)^{\alpha} \norme{u}_{L^{2}(\Omega)}^{1-\alpha}.
\end{equation}
We conclude the proof of \eqref{eq:spectraldirichlet} by gathering \eqref{eq:spectraldirichletRestrictedToballL2App} and \eqref{eq:FinStep2}. This ends the proof of \Cref{thm:spectralopendirichlet}.
\end{proof}

\subsection{Adaptation of the Lebeau-Robbiano's method}

With the help of the spectral estimate of \Cref{thm:spectralopendirichlet}, one can now prove \Cref{thm:observability_main_result}. Notice that by using the spectral estimate \eqref{eq:spectraldirichlet}, the proof of the observability estimate  \eqref{eq:Observability_quantitative} is by now quite classical and originally comes from the Lebeau-Robbiano's method for obtaining the null-controllability of the heat equation starting from a spectral estimate, see \cite{LR95} and \cite[Section 6]{LRL12}. This strategy was latter extended by Miller in \cite{Mil10}. 
Note that it has to be adapted to the particular functional setting of the spectral estimate \eqref{eq:spectraldirichlet}. Here we follow the arguments of \cite[Proof of Theorem 3 in Section 4]{BM23} and \cite[Remark 2 in Section 2]{AEWZ14}. 
\begin{proof}[Proof of \Cref{thm:observability_main_result}]
From \eqref{eq:spectraldirichlet} and \cite[Corollary 4.1]{BM23}, we have the following result. There exists $A>0,\ C>0$ such that that for all $0< t_1 <t_2\leq T$, $u_0 \in L^{2}(\Omega)$, the solution $u$ of \eqref{eq:Heat_equation} satisfies
\begin{equation}\label{eq:firsttelescoping}
 e^{- \frac{A}{(t_2 -t_1)}} \| u(t_2,\cdot) \|_{L^2(\Omega)} - e^{- \frac{2A}{(t_2- t_1)}} \| u(t_1,\cdot) \|_{L^2(\Omega)} \leq C \int_{t_1} ^{t_2}  \sup_{x \in \omega} | u(s,x)| ds.
 \end{equation}
We now define $l = T/2,\ l_1 = 3T/4,\  l_{m+1} -l = 2^{-m} (l_1-l),\ m \geq 1$. We apply \eqref{eq:firsttelescoping} with $t_1 = l_{m+1}$ and $t_2 = l_m$, 
 \begin{equation*}
 e^{-\frac{A} {l_m - l_{m+1}} } \| u(l_m,\cdot)\|_{L^2(\Omega)} -  e^{-\frac{2A} {(l_m - l_{m+1})} } \| u(l_{m+1},\cdot)\|_{L^2(\Omega)}\leq C \int_{l_{m+1}}^{l_m} \sup_{x \in \omega} | u(s,x)| ds. 
 \end{equation*}
Noticing that $\frac 2 {(l_m - l_{m+1})}=\frac 1 {( l_{m+1} - l_{m+2})}$,  we get
 \begin{equation}\label{eq:secondtelescoping}
 e^{-\frac{A} {l_m - l_{m+1}} } \| u(l_m,\cdot)\|_{L^2(\Omega)} -  e^{-\frac{A} {l_{m+1} - l_{m+2}} } \| u(l_{m+1},\cdot)\|_{L^2(\Omega)}\leq C \int_{l_{m+1}}^{l_m}  \sup_{x \in \omega} | u(s,x)| ds ,
 \end{equation}
then summing the telescopic series~\eqref{eq:secondtelescoping}, and using that 
$ \lim_{m\rightarrow + \infty} e^{-\frac{A} {l_{m+1} - l_{m+1}} } =0, $ we get
 $$e^{-\frac{A} {l_1 - l_{2}} } \| u(l_1,\cdot)\|_{L^2(\Omega)} \leq C \int_{l}^{l_1}  \sup_{x \in \omega} | u(s,x)| ds. $$
From the definition of $l, l_1,l_2$ we immediately deduce
 $$ \| u(l_1,\cdot)\|_{L^2(\Omega)} \leq C e^{C/T} \int_{0}^{T}  \sup_{x \in \omega} | u(s,x)| ds,  $$
 then from the $L^2-L^{\infty}$ regularizing effect of the heat equation, see for instance \cite[Proposition 3.5.7]{CH98}, we obtain
 $$ \| u(T,\cdot)\|_{L^{\infty}(\Omega)} \leq C T^{-n/4} \| u(l_1,\cdot)\|_{L^2(\Omega)} \leq C e^{C/T} \int_{0}^{T}  \sup_{x \in \omega} | u(s,x)| ds,  $$
 so the quantitative observability estimate \eqref{eq:Observability_quantitative}.
\end{proof}

\section{Measurable sets with lower Hausdorff dimension}
\label{sec:MeasurablesetsLower}

This section is devoted to study the observability of \eqref{eq:Heat_equation} in some particular cases when the observable sets $\omega$ has a Hausdorff dimension lower than $n-1$. In fact, we exemplify cases of zero dimension in arbitrarily large ambient dimension.

\subsection{Observability from a finite subset}

To set the stage, we recall that in the case of the one-dimensional heat equation posed on $\Omega =(0,1)$, Dolecki fully described in \cite{Dol73} the situation when $\omega=\{x_0\}$ in terms of certain algebraic properties of the number $x_0$. Dolecki’s result in fact precisely describes the gap between nodal sets of the eigenfunctions and observation sets. Namely, observation points are characterized by their not being too close (on the wavelength scale) to a nodal set, which in Dolecki’s formulation means $x_0$ is not well-approximated by rationals numbers.  By analogy with respect to the one-dimensional case, the natural extension to the $n$-dimensional case should be: “$\omega$ is not well-approximated by the nodal sets of the Laplace eigenfunctions.” Such remark already appears in \cite[Section 1.3]{Let19}. In \cite{Sam15}, Samb precised Dolecki's statement by showing that, in the case when $x_0$ is an algebraic number of order $d>1$, the observability inequality holds at any time $T>0$. He also provided an estimation of the observability cost, which we record for later use.

\begin{theorem}[\cite{Sam15}]\label{thm:1D_Samb}
Let $x_0 \in (0,1)$ be an irrational algebraic number of order $d>1$. There exists a positive constant $C>0$ such that for any $T>0$, for every $u_0 \in L^2(\Omega)$, the solution $u$ of \eqref{eq:Heat_equation} satisfies
$$\|u(T,\cdot)\|^2_{L^{\infty}(0,1)} \leq Ce^{\frac CT} \int_0^T |u(t,x_0)|^2 dt.$$
\end{theorem}

From Samb's proof, it is clear that the result holds for every diophantine number (i.e. non Liouville number) since it  only involves that an irrational algebraic number is a diophantine number. We will not use this fact in the rest of the article, so we will stick to Samb's formulation of \Cref{thm:1D_Samb}. Note also that the observability in \Cref{thm:1D_Samb} is a little bit different from the one mainly used in this paper, i.e. \eqref{eq:Observability_definition}.

A straightforward dimension-counting argument shows that such a result cannot hold on domains for which the dimension of the eigenspaces $E_r(-\Delta)$ is unbounded as $r \to \infty$, in particular when $\Omega = (0,1)^n$ for $n \ge 2$.
\begin{proposition}\label{prop:MultiD_finite_observable_set}
Let $\Omega=(0,1)^n$ with $n\geq 2$ and $T>0$. For any finite subset $\omega \subset \Omega$, the heat equation \eqref{eq:Heat_equation} is not observable from $\omega$ at time $T$.
\end{proposition}
This proposition shows that the results of \cite{Dol73} and \cite{Sam15} are very specific to the one-dimensional case. This difference with the multi-dimensional case comes from the fact that, when $n\geq 2$, the degree of degeneracy of the energy levels of the operator $-\Delta$ tends to $+\infty$. Let us recall that in this case, the spectrum of $-\Delta$ is given by 
\begin{equation}\label{eq:Spectrum_Laplacian}
    \text{Sp}(-\Delta) =\left\{ r_1^2 +\cdots + r_n^2; \ (r_1, \cdots, r_n) \in (\mathbb N^*)^n\right\}
\end{equation}
and for all $r \in \text{Sp}(-\Delta)$, the eigenspace $E_r(-\Delta)$ associated to $r$ is
$$E_r(-\Delta) = \text{Span}_{\mathbb C}\left\{\varphi_{r_1} \otimes \varphi_{r_2} \otimes \cdots \otimes \varphi_{r_n}; \ r_1^2+\cdots+r_n^2=r\right\},$$
with dimension equal to the sum of squares counting function
$$s_n(r)=\#\left\{(r_1, \cdots, r_n) \in (\mathbb N^*)^n; \ r_1^2+\cdots +r_n^2=r\right\}.$$
If $n \ge 2$, then 
    \begin{equation}\label{e:rdn} \limsup_{r \to \infty} s_n(r) = \infty.\end{equation} 
Indeed, for $n=2$ one can explicitly compute $s_2(5^{2r-1}) = r$ by \cite[pp.140-142]{Bei64}. For $n\ge 2$, of course $s_n(r) \ge s_2(r)$ hence \eqref{e:rdn} holds.

\begin{proof}[Proof of Proposition~\ref{prop:MultiD_finite_observable_set}]
Let $\omega=\{x_1, \cdots, x_{N}\}\subset (0,1)^n$. By \eqref{e:rdn}, we can find $r$ such that $s_n(r) \ge N+1$. Since $\text{dim}(E_r(-\Delta))=s_n(r)$, it follows that the linear application $T : f \in E_r(-\Delta) \longmapsto (f(x_1), ..., f(x_{N})) \in \R^N$ is not one-to-one. This provides an eigenfunction $\psi$, associated to the eigenvalue $r$, such that 
$$\psi(x_1)=\cdots=\psi(x_{N})=0.$$
Notice that for all $t \in [0,T]$, $e^{t \Delta} \psi=e^{-rt} \psi$. This leads to 
$$\int_0^T \sup_{x \in \omega}|e^{t\Delta} \psi(x)| dt= \int_0^T e^{-rt} \sup\{|\psi(x_1)|, ..., |\psi(x_{N})|\} dt =0.$$ Since $\|\psi\|_{L^{\infty}(\Omega)} \neq 0$, the observability of \eqref{eq:Heat_equation} does not hold from $\omega$ at time $T>0$. This ends the proof of Proposition~\ref{prop:MultiD_finite_observable_set}.
\end{proof}

Note that the proof of \Cref{prop:MultiD_finite_observable_set} crucially uses the fact that the dimension of the eigenspaces $E_r(-\Delta)$ is unbounded as $r \to \infty$, a property that generically does not hold. According to the proof of \cite[Theorem 2.79]{Cor07}, the following more general result holds.
\begin{proposition}\label{prop:MultiD_finite_observable_set_Coron}
Let $\Omega$ be a bounded open set of $\R^n$, $n \geq 2$. For any finite subset $\omega \subset \Omega$, the heat equation \eqref{eq:Heat_equation} is not observable from $\omega$ at time $T$.
\end{proposition}
In light of this negative result, we would like to see, in higher dimensions, how small specially constructed observation sets can be, knowing that a collection of finitely many points is a priori forbidden. As we generate sets of even smaller dimension we make the following comment. Since the eigenfunctions of $-\Delta$ are continuous, $\Pi_\Lambda u$ is always continuous which entails
    \begin{equation}\label{e:closure} \sup_{x \in E} \abs{\Pi_\Lambda u(x)} = \sup_{x \in \overline{E}} \abs{\Pi_\Lambda u(x)}.\end{equation}
Therefore, to generate nontrivial examples, we should restrict our attention to sets whose \textit{closure} is low-dimensional (or to closed sets). An example to motivate this restriction is the zero-dimensional set $\omega=\mathbb Q^n \cap \Omega$ for which the spectral inequality \eqref{eq:spectraldirichlet} trivially holds by \eqref{e:closure}.

In fact, we will construct some families of low-dimensional observation sets by taking the Cartesian product of other sets of low dimension. The first, simple principle along these lines is we can take the Cartesian product of spectral inequalities; see Proposition \ref{prop:spectral_cartesian_product2} below. The second, more involved argument is that we can take Cartesian product of an observability estimate (with good control cost) and a spectral inequality, see \Cref{prop:Obversability+spectral} below. 

Thus, our starting point in the following improved spectral inequality in one dimension, which in fact has an extremely elementary proof.

\subsection{One-dimensional case}
Our starting point is the following very simple Bang-type estimate for analytic functions. It is a simplified version of \cite[Estimate (4.1)]{NSV04} or \cite[Lemma 4.7]{JM22} with one small modification, replacing the Lebesgue measure with the following quantity related to transfinite diameter, see \cite{Pom92}. For any set $E \subset [0,1]$, define 
	\begin{equation}\label{e:gamma} \gamma_k(E) = \sup_{\{x_i\}_{i=1}^k \subset E} \min_{i} \prod_{j \ne i} |x_i-x_j|.\end{equation}
The following one-dimensional Remez estimate holds. 
\begin{proposition}\label{prop:prop}
Let $(M_\ell)_{\ell \geq 0} \subset (0,\infty)^{\N}$ and suppose there exists $\ell_0 \in \N$ such that $M_{\ell_0}/\ell_0! \le \frac 12$. Then, for any $f \in C^\infty([0,1])$ satisfying
	\begin{equation}
 \label{e:d} \sup_{[0,1]}|f^{(l)}| \le M_\ell \sup_{[0,1]} |f|, \qquad \ell \geq 0,\end{equation}
and any $E \subset [0,1]$,
	\begin{equation}\label{e:remez} \sup_{[0,1]}|f| \le \frac{ 4\ell_0}{\gamma_{\ell_0}(E)} \sup_{E} |f|.\end{equation}
\end{proposition}

\begin{proof}
The proof relies chiefly on the following Lagrange interpolation formula (see for instance \cite[Chapitre 2, Section 1.2]{Dem16}). For any $x$, $x_1<\cdots < x_k$ in $[0,1]$, there exists $\xi$ in $[0,1]$ such that
	\begin{equation}\label{e:interp} f(x) = \sum_{i=1}^k f(x_i) \prod_{j \ne i} 
\frac{(x-x_j)}{(x_i-x_j)} + \frac{f^{(k)}(\xi)}{k!} \prod_{j} (x-x_j).\end{equation}
Choosing $\{x_i\}_{i=1}^k \subset E$ such that $2\min_i \prod_{j \ne i} \abs{x_i-x_j} \ge \gamma_k(E)$, it is immediate that
	\[ \left \vert \sum_{i=1}^n f(x_i) \prod_{j \ne i} \frac{(x-x_j)}{(x_i-x_j)} \right \vert \le  \frac{2k \cdot \sup_{E} |f| }{\gamma_k(E)} \]
To estimate the second term in \eqref{e:interp} we simply apply \eqref{e:d} to obtain 
	\[ \left\vert \frac{f^{(k)}(\xi)}{k!} \prod_{j} (x-x_j) \right\vert \le  \frac{M_k}{k!} \sup_{[0,1]}|f|.\] 
which is less than $\frac 12 \sup_{[0,1]} \abs{f}$ for $k=\ell_0$ hence \eqref{e:remez} holds.
\end{proof}

From Proposition \ref{prop:prop}, we can derive a spectral estimate for many sets if we can control $\gamma_k(\omega)$. First, if $\omega$ has positive logarithmic capacity (which in particular holds in $\text{dim}_{\mathcal H}(\omega) > 0$) then $\gamma_k(\omega) \gtrsim m^{k}$, see \cite[Proposition 9.6]{Pom92}. More generally though, we obtain an effective spectral inequality if $\omega$ satisfies
    \begin{equation}\label{e:gammak} \limsup_{k \to \infty} \frac{\log \left(-\log \gamma_k(\omega)\right)}{\log k} <2.\end{equation}
Notice that in the case of positive dimension this quantity equals $1$. 
\begin{corollary}\label{cor:one-dim-gamma}
    If $\Lambda \ge 1$ and $\omega$ satisfies \eqref{e:gammak}, then there exists $C>0$ and $\beta \in (0,1)$ such that
    \begin{equation}
    \label{eq:spectraltransfinitedimameter}
        \sup_{x \in [0,1]} \abs{v(x)} \le \exp\left( C \Lambda^\beta \right) \sup_{x \in \omega} |v(x)|,\qquad v=\sum_{\substack{k \in \mathbb N: \\ k^2 \le \Lambda}} v_k \sin(2 \pi k x).
    \end{equation}
\end{corollary}
\begin{proof}
By the one-dimensional Bernstein inequality, $v$ satisfies \eqref{e:d} with $M_\ell = (C\sqrt{\Lambda})^\ell$. Therefore, by Stirling's formula, $(C\sqrt{\Lambda})^\ell / \ell! \le 1/2$ as soon as $\ell = A \sqrt{\Lambda}$ for some absolute constant $A \gtrsim C$. Therefore, applying \eqref{e:remez}, it only remains to estimate the constant
    \[ \frac{4 \ell_0}{\gamma_{\ell_0}(\omega)}, \quad \text{with} \ \ell_0 = A \sqrt{\Lambda}.\]
However, \eqref{e:gammak} implies that 
    \[ \frac{4 k}{\gamma_{k}(\omega)} \le \exp( C k^{\alpha}) \]
for some $\alpha < 2$ and the proof is concluded by taking $k=\ell_0=A\sqrt{\Lambda}$ and $\beta = \frac \alpha 2$.
\end{proof}
We now introduce a family of zero-dimensional sets, indexed by $\alpha > 0$,
\begin{equation}
\label{eq:defomegaalpha}
    \omega_\alpha = \left\{ i^{-\alpha}: i=1,2,\ldots \right\}.
\end{equation}
 We have the following crude estimate for $\gamma_k(\omega_\alpha)$ (recall $\gamma_k$ from \eqref{e:gamma}) which will suffice for our purposes, see \Cref{cor:spec-omega-alpha} below. 
\begin{lemma}\label{l:gamma-omega-alpha}
For each $\alpha >0$ there exists $c_1,c_2>0$ such that
    \begin{equation}\label{e:gamma-omega-alpha} \gamma_k(\omega_\alpha) \ge c_1 e^{c_2 k \log k}.\end{equation}
\end{lemma}
\begin{proof}
    By simply selecting the first $k$ terms of the sequence $(i^{-\alpha} )_{i \in \mathbb N^*}$, we have
        \[ \gamma_k(\omega_\alpha) \ge \inf_{i=1,\ldots,k} \prod_{\substack{j=1 \\ j \ne i}}^k \abs{ j^{-\alpha}-i^{-\alpha}}.\]
    Rewrite 
        \[ \prod_{\substack{ j=1 \\ j \ne i} }^{k} \abs{ j^{-\alpha}-i^{-\alpha} } = \prod_{\substack{ j=1 \\ j \ne i} }^{k} j^{-\alpha} i^{-\alpha} \abs{ j^{\alpha} - i^{\alpha} }.\]
    Obviously $\abs{ j^{\alpha} - i^{\alpha} }$ is minimized when $j$ is close to $i$ so by the mean value theorem $\abs{ j^{\alpha} - i^{\alpha} } \gtrsim i^{\alpha-1}$. Therefore, 
        \[ \gamma_k(\omega_\alpha) \gtrsim \inf_{i=1,\ldots,k} \prod_{j=1}^k j^{-\alpha} i^{-1} = \inf_{i=1,\ldots,k} (k!)^{-\alpha} i^{-k}.\]
    Clearly this infimum is attained when $i=k$ and then the form \eqref{e:gamma-omega-alpha} follows by Stirling's formula.
\end{proof}

\subsection{Spectral inequalities for Cartesian products}

Roughly speaking, the goal of this part consists in proving that spectral inequalities respect taking Cartesian products. Then we generate observable sets of zero dimension by taking Cartesian products of $\omega_{\alpha}$, defined above in \eqref{eq:defomegaalpha}.

To succinctly describe the spectral inequalities of the form \eqref{eq:spectraldirichlet}, we say $\mathrm{spec}(\Omega,\omega,C,\beta)$ holds if
$$ \|\Pi_\Lambda u\|_{L^{2}(\Omega)} \leq C \exp(C\Lambda^\beta) \sup_{x \in \omega} |\Pi_\Lambda u (x) |,\quad \forall \Lambda >0, \forall u \in L^2(\Omega),$$
where $\Pi_\Lambda$ is the spectral projection defined in \eqref{e:spec-proj}. 

\begin{proposition}
\label{prop:spectral_cartesian_product2}
Let $\Omega_j\subset \R^{n_j}$, $j=1,2$ and $\omega \subset \Omega_1 \times \Omega_2$. Let $\omega_2=\pi_2(\omega)$ be the projection of $\omega$ on $\Omega_2$. For each $y \in \omega_2$, we denote by 
    \[ \omega_y:=\pi_1(\omega\cap (\Omega_1 \times \{y\}))\] 
the horizontal slices of $\omega$. If there exists $C,\beta>0$ such that $\mathrm{spec}(\Omega_2,\omega_2,C,\beta)$ holds and for every $y \in \omega_2$, $\mathrm{spec}(\Omega_1,\omega_y,C,\beta)$ holds, then there exists $C'$ such that
    \[ \mathrm{spec}{(\Omega_1 \times \Omega_2, \omega,C',\beta)} \quad \text{holds.}\]
\end{proposition}
When $\omega=\omega_1 \times \omega_2$, notice $\omega_y=\omega_1$ for all $y \in \omega_2$, thus in this case we have the corollary
    \[ \Big( \mathrm{spec}(\Omega_j,\omega_j,C,\beta)\quad \text{holds},\qquad j=1,2\Big) \implies \mathrm{spec}(\Omega_1 \times \Omega_2,\omega_1\times\omega_2,C',\beta)\quad \text{holds}.\]

\begin{proof}
 Let $\Lambda >0$ and $u \in L^2(\Omega_1 \times \Omega_2)$. Let $(\varphi_k^1)_{k\in \nn}$ (resp. $(\varphi_k^2)_{k\in \nn}$) be an orthonormal basis of $L^2(\Omega_1)$ (resp. $L^2(\Omega_2)$) formed by eigenfunctions of $-\Delta_{\Omega_1}$ (resp. $-\Delta_{\Omega_2}$). Then $$\Pi_\Lambda u (x,y)=\sum_{\lambda_p^1+\lambda_q^2 \leq \Lambda} C_{p,q} \varphi^1_p(x)\varphi^2_q(y).$$
    The key observation is that for each $y \in \Omega_2$, define the function $w_y$ on $\Omega_1$ by $w_y = \Pi_\Lambda u(\cdot,y)$. We can explicitly write 
        \[ w_y(x) = \sum_{\lambda_p^1 \le \Lambda} c_{y,\lambda_p^1} \varphi_p^1(x), \quad c_{y,\lambda_p^1} = \sum_{\lambda_q^2 \le \Lambda -\lambda_p^1} C_{p,q} \varphi_q^2(y),\]
    to see first that the spectral inequality on $\Omega_1$ applies to $w_y$ to obtain for each $y \in \omega_2$,
        \[ \norm{w_y}_{L^2(\Omega_1)} \le C\exp (C \Lambda^\beta) \sup_{x \in \omega_y} \abs{w_y(x)} \le C\exp (C \Lambda^\beta) \sup_{(x,y) \in \omega} \abs{\Pi_\Lambda(x,y)},\]
    where the last inequality followed from $w_y(x)=\Pi_\Lambda u(x,y)$ and the fact that $\omega_y \subset \omega$ for $y \in \omega_2$.
    Therefore, if we can show
        \begin{equation}\label{e:goal-wy} \int_{\Omega_2} \norm{w_y}^2_{L^2(\Omega_1)} \,dy \le C'\exp(C'\Lambda^\beta) \sup_{y\in \omega_2}  \norm{w_y}_{L^2(\Omega_1)}^2,\end{equation}
    then the proof will be concluded by Fubini's theorem. Since the functions $y \mapsto c_{y,\lambda_p^1}$ are also linear combinations of eigenfunctions with eigenvalues less than $\Lambda$, we can apply the spectral inequality on $\Omega_2$ to each of them. Furthermore, by Parseval's identity
        \begin{equation}\label{e:parseval-w}\norm{w_y}^2_{L^2(\Omega_1)}  = \sum_{\lambda_p^1 \le \Lambda} \abs{c_{y,\lambda_p^1}}^2. \end{equation} 
    Therefore, by Fubini's theorem,
    \[ \int_{\Omega_2} \norm{w_y}^2_{L^2(\Omega_1)} \,dy = \sum_{\lambda_p^1 \le \Lambda} \int_{\Omega_2} \abs{c_{y,\lambda_p^1}}^2 \, dy \le C \exp (C \Lambda^\beta) \sum_{\lambda_p^1 \le \Lambda}\sup_{y \in \omega_2}  \abs{c_{y,\lambda_p^1}}^2. \]
    To bring the supremum outside the summation, we must pay a price of the cardinality of the sum. However, by the Weyl's law, this cardinality is polynomial (depending on $\Omega_2$) in $\Lambda$, so that
        \[  \sum_{\lambda_p^1 \le \Lambda}\sup_{y \in \omega_2}  \abs{c_{y,\lambda_p^1}}^2 \le C \Lambda^m  \sup_{\lambda_p^1 \le \Lambda}\sup_{y \in \omega_2}  \abs{c_{y,\lambda_p^1}}^2 \le C \Lambda^m \sup_{y \in \omega_2} \sum_{\lambda_p^1 \le \Lambda} \abs{c_{y,\lambda_p^1}}^2.\]
    Therefore, combining the above two displays, using \eqref{e:parseval-w}, and choosing $C'$ appropriately to absorb the polynomial growth, we obtain \eqref{e:goal-wy}.
\iffalse
    Let $\Lambda >0$ and $u \in L^2(\Omega_1 \times \Omega_2)$. Let $(\varphi_k^1)_{k\in \nn}$ (resp. $(\varphi_k^2)_{k\in \nn}$) be an orthonormal basis of $L^2(\Omega_1)$ (resp. $L^2(\Omega_2)$) formed by eigenfunctions of $-\Delta_{\Omega_1}$ (resp. $-\Delta_{\Omega_2}$). Then $$\Pi_\Lambda u (x,y)=\sum_{\lambda_p^1+\lambda_q^2 \leq \lambda} C_{p,q} \varphi^1_p(x)\varphi^2_q(y).$$
    The key observation is that for any $x \in \Omega_1$ and $y \in \Omega_2$,
        \[ v_x = \Pi_\Lambda u(x,\cdot), \quad w_y = \Pi_\Lambda u(\cdot,y)\]
    are each linear combinations of eigenfunctions of energy level less than $\Lambda$ to which the spectral inequalities apply. Thus, for any $x \in \Omega_1$, using that $\mathrm{spec}(\Omega_2,\omega_2,C,\beta)$ holds, 
        \[ \sup_{y \in \Omega_2} |\Pi_\Lambda u(x,y)| =  \sup_{y \in \Omega}|v_x(y)| \le C \exp( C \Lambda^\beta) \sup_{\omega_2} \abs{v_x} = C \exp( C \Lambda^\beta)\sup_{y \in \omega_2} \abs{\Pi_\Lambda u(x,y)} .\]
    And for any $y \in \omega_2$, $x \in \Omega_1$, using $\mathrm{spec}(\Omega_1,\omega_y,C,\beta)$ holds, 
        \[ \abs{\Pi_\Lambda u(x,y)} = \abs{w_y(x)} \le  C \exp( C \Lambda^\beta) \sup_{z \in \omega_y} \abs{w_y(z)} \le C \exp( C \Lambda^\beta) \sup_{\omega} \abs{\Pi_\Lambda u}.\]
    Combining these two displays and using the fact that $x \in \Omega_1$ was arbitrary concludes the proof. 
\fi
\end{proof}

By combining \Cref{prop:prop} and \Cref{l:gamma-omega-alpha}, we obtain the following corollary.
\begin{corollary}\label{cor:spec-omega-alpha}
For any $\beta \in \left(\tfrac 12,1\right)$ there exists $C>0$ such that
    $\mathrm{spec}((0,1),\omega_\alpha,C,\beta)$ holds.
\end{corollary}

Hence the following result holds by using \Cref{cor:spec-omega-alpha}, \Cref{prop:spectral_cartesian_product2} and the so-called Lebeau-Robbiano method, see for instance \cite{Mil10}.

\begin{corollary}\label{cor:obs-omega-alpha-spec}
    Let $\alpha>0$, $\beta \in (\frac{1}{2},1)$ $n \ge 1$. Then, there exists $C>0$ such that $\mathrm{spec}((0,1)^n, \omega_{\alpha}^n, C, \beta)$ holds. Therefore, there exists $C>0$ such that for every $T>0$, for every $u_0 \in L^2(\Omega)$, the solution $u$ of \eqref{eq:Heat_equation} satisfies
    \begin{equation*}
        \|u (T,\cdot)\|_{L^{\infty}(\Omega)} \leq C \exp(C T^{\beta/(1-\beta)}) \int_0^T \sup_{x \in \omega} |u(t,x)|dt.
    \end{equation*}
\end{corollary}

In fact the Cartesian product $\omega_\alpha \times \ldots \times \omega_\alpha$ is countable and therefore must have Hausdorff dimension zero. We note that $\omega_\alpha$ accumulates at a boundary point, but one could shift $\omega_\alpha$ to translate this accumulation point to the interior and obtain the same result. Indeed, $\gamma_k(\cdot)$ is translation invariant, so Corollaries \ref{cor:spec-omega-alpha} and \ref{cor:obs-omega-alpha-spec} hold independent of the position of the observation set. 

\subsection{Observability inequalities for Cartesian products}

Roughly speaking, the goal of this part consists in proving that one can take a Cartesian product of a spectral estimate and an observability inequality.
\begin{proposition}\label{prop:Obversability+spectral}
Let $T>0$, $\omega_j \subset \Omega_j \subset \R^{n_j}$, $j=1,2$. Let $\omega_1 \subset \Omega_1$ and $\omega_2 \subset \Omega_2$. Assume that\\
\indent $\bullet$ there exists $C_2>0$ and $\beta \in (0,1)$ such that $\mathrm{spec}(\Omega_2,\omega_2,C_2,\beta)$ holds,\\
\indent $\bullet$ there exists $C_1>0$ such that for $\alpha = \frac{\beta}{1-\beta}$, for every $v_0 \in L^2(\Omega_1)$, the solution $v$ of \eqref{eq:Heat_equation} in $\Omega_1$ satisfies
$$\forall 0 <\tau \le T, \  \|v(\tau, \cdot)\|^2_{L^{2}(\Omega_1)} \leq C_1 \exp \left( \tfrac{C_1}{\tau^\alpha} \right) \int_0^\tau \sup_{x \in \omega_1} |v(t,x)|^2 dt.$$
Then there exists $C>0$ such that
\begin{equation}
    \label{eq:ObsCartesianProducts}
    \|u(T, \cdot)\|_{L^{2}(\Omega_1 \times \Omega_2)}^2 \leq C \exp \left( \tfrac{C}{T^\alpha} \right) \int_0^T \sup_{\omega_1 \times \omega_2} |u(t,x)|^2 dt,
\end{equation}
where $u$ is the solution of \eqref{eq:Heat_equation} on $\Omega_1 \times \Omega_2$.
\end{proposition}

Note that \eqref{eq:ObsCartesianProducts} differs from the functional spaces of the observability notion \eqref{eq:Observability_definition} in \Cref{def:ObsL2L2Linfty} but this is because of the second assumption. By using the $L^2-L^{\infty}$ regularizing effect of the heat equation, at least the left hand side of \eqref{eq:ObsCartesianProducts} can be transformed into $\|u(T, \cdot)\|_{L^{\infty}(\Omega_1 \times \Omega_2)}^2$.

The following proof is adapted from \cite[Theorem~1.1]{Sam15} who was considering the case of $\{x_0\} \times \omega \subset (0,1)^2$ where $x_0$ is an algebraic number of degree $d>1$ and $\omega$ an open interval of $(0,1)$. A generalization of this type of result, encompassing the case where the minimal time of null-controllability $T(x_0)$ is strictly positive, has been recently considered in the recent preprint \cite{AKBGBMdT24}.
\begin{proof} 
The proof consists of applying the observability estimate on $\omega_1$ before using spectral estimates and then performing a Lebeau-Robbiano argument type.
Let $0<\tau\leq 1$, $u_0 \in L^2(\Omega_1 \times \Omega_2)$ and $\Lambda>0$ to be chosen later. Let us write the space variable $z=(x,y)\in \Omega_1\times \Omega_2$ so that $\Delta = \Delta_{\Omega_1} + \Delta_{\Omega_2}$.
As before, let $(\varphi_k^j)_{k\in \nn}$ be an orthonormal basis of $L^2(\Omega_j)$, for $j=1,2$, consisting of eigenfunctions of $-\Delta$ with Dirichlet boundary conditions. If $u$ is the solution of \eqref{eq:Heat_equation}, we can therefore write 
$$\Pi_{\Lambda}^2 u(t, x,y) = \sum_{\lambda^2_k \leq \Lambda} u_k(t, x) e^{-t \lambda^2_k}\varphi^2_k(y),$$
where the $u_k$ are solutions of the heat equation on $\Omega_1$ and $\Pi^j_\Lambda$ is the spectral projection \eqref{e:spec-proj} on $\Omega_j$.

So for $y \in \Omega_2$ fixed, $v_y(t, x):=\sum_{\lambda^2_k \leq \lambda} u_k(t, x) e^{-\tau \lambda^2_k}\varphi^2_k(y)$ is also a solution of the heat equation on $\Omega_1$, whence we have the observability inequality
    \begin{equation}\label{e:obs-vy} \|v_y(\tau,\cdot)\|^2_{L^2(\Omega_1)}  \leq C_1\exp\left( \tfrac {C_1}{\tau^\alpha}\right ) \int_{\frac{\tau}{2}}^\tau \sup_{x \in \omega_1}|v_y(t, x)|^2 dt\end{equation}
for any $y \in \Omega_2$. Furthermore, by the Cauchy-Schwarz inequality, for $x \in \omega_1$ and $t \in \left( \tfrac \tau 2, \tau \right]$,
     \[ |v_y(t,x)|^2 \le \left( \sum_{\lambda_k^2 \le \Lambda} \abs{u_k(t,x)}^2 e^{-2t \lambda_k^2} \right) \left( \sum_{\lambda_k^2 \le \Lambda} \abs{\varphi_k^2(y)}^2 \right).\]
By the Weyl's law, integrating the second factor over $\Omega_2$ contributes polynomial blow-up in $\Lambda$. This is of no consequence since the first factor, by Parseval's identity and the spectral inequality, satisfies
    \[ \sup_{x \in \omega_1} \sum_{\lambda^2_k \leq \Lambda} \left|u_k(t, x)\right|^2 e^{-2t \lambda^2_k} =
\sup_{x \in \omega_1} \int_{\Omega_2} \left|\Pi^2_{\Lambda} u(t, x, y) \right|^2 dy\]
\[\leq C_2 \exp\left (C_2 \Lambda^\beta \right) \sup_{(x,y) \in \omega_1 \times \omega_2} \left|\Pi^2_{\Lambda} u(t, x, y) \right|^2.\]
Integrate \eqref{e:obs-vy} over $y \in \Omega_2$ and combine it with the above display and the identity
 \[ v_y(\tau,x) = \Pi^2_{\Lambda} u(\tau,x,y)\]
 to obtain the low-frequency observability inequality
\begin{equation}\label{e:low-freq-obs} \norm{\Pi^2_\Lambda u(\tau,\cdot)}_{L^2(\Omega_1 \times \Omega_2)}^2 \le C_3 \exp\left( \tfrac{C_1}{\tau^\alpha} + C_3 \Lambda^\beta \right) \int_{\frac \tau 2}^\tau \sup_{(x,y) \in \omega_1 \times \omega_2} \left|\Pi^2_{\Lambda} u(t, x, y) \right|^2 \, dt. \end{equation}
($C_3$ is introduced in order to absorb the polynomial growth in $\Lambda$).

It is classical that since $u^+ :=(I-\Pi^2_\Lambda) u$ satisfies the heat equation on $\Omega=\Omega_1 \times \Omega_2$, we have from the $L^2-L^{\infty}$ regularizing effect of the heat equation, see for instance \cite[Proposition 3.5.7]{CH98},
\begin{equation}\label{e:reg}\forall t>\frac \tau 2, \quad \sup_{\Omega}\left|u^+(t, \cdot)\right| \leq \frac C{\tau^{\frac n4}} \left\| u^+\left(\frac \tau 4, \cdot\right) \right\|_{L^2(\Omega)} \leq \frac C{\tau^{\frac n4}} e^{-\Lambda \frac \tau 4} \|u_0 \|_{L^2(\Omega)},\end{equation}
for some positive constant $C>0$.
So by applying the triangle inequality, \eqref{e:low-freq-obs}, triangle inequality again, followed by \eqref{e:reg} to absorb the error, we obtain that for all $0<\tau \leq 1$ and $\Lambda>0$
\begin{align*}
    &\|u(\tau,\cdot)\|^2_{L^2(\Omega)} \\
    & \leq C_3 \exp\left( \tfrac {C_1}{\tau^\alpha} + C_3 \Lambda^\beta \right) \left[ \int_{\frac \tau 2}^{\tau} \sup_{\omega_1 \times \omega_2}|u(t,\cdot)|^2 dt +\tau^{- \frac n4} C \exp\left( -\tfrac\tau 4 \Lambda \right) \|u_0\|^2_{L^2(\Omega)} \right].
\end{align*}
One can find a positive parameter $\gamma>0$ such that $ \gamma \geq 10\left( C_1+C_3\gamma^\beta \right)$.
Let us choose then $\Lambda = \frac{\gamma}{\tau^{\alpha+1}}$, noting the key relationship $\beta(\alpha +1) = {\alpha}$, so that
    \[ \exp\left( \tfrac {C_1}{\tau^\alpha} + C_3 \Lambda^\beta \right) \le \exp \left( \tfrac \gamma{10 \tau^\alpha} \right), \quad \tau^{-\frac{n}{4}}\exp\left( -\tfrac\tau 4 \Lambda \right) \le C_5 \exp\left( -\tfrac{\gamma}{5\tau^\alpha}\right). \]
Therefore,
\begin{equation}\label{eq:Lebeau_robbiano_strategy}
    f(\tau) \|u(\tau,\cdot)\|^2_{L^2(\Omega)} \leq \int_{\frac \tau2}^\tau \sup_{\omega_1 \times \omega_2}|u(t,\cdot)|^2 dt +f(\eta \tau ) \|u_0\|^2_{L^2(\Omega)},
\end{equation}
with $f(\tau)=C_6 \exp\left(-\frac{\gamma}{10\tau^\alpha} \right)$, and $\eta = 2^{- \frac 1\alpha}<1$. 
Let us assume that $0< T \leq 1$ and define the following sequence:
$$T_0=T \quad \text{and} \quad \forall k \in \N, \ T_{k+1}= T_k-(1-\eta)\eta^{k} T.$$
By applying \eqref{eq:Lebeau_robbiano_strategy} with initial data $u_0 = u(T_{k+1})$ and running until time $T_k$ ($\tau$  then represents the elapsed time $T_k - T_{k+1} = (1-\eta)\eta^{k+1}T$), we obtain 
\begin{align*}
&f\left((1-\eta)\eta^{k}T\right) \| u(T_{k},\cdot)\|^2_{L^2(\Omega)} \\ 
& \leq \int_{T_{k+1}}^{T_k} \sup_{\omega_1 \times \omega_2}|u(t,\cdot)|^2 dt +f\left((1-\eta)\eta^{k+1}T\right) \|u(T_{k+1},\cdot)\|^2_{L^2(\Omega)}.
\end{align*}
Since $f\left( (1-\eta)\eta^k T\right) \to 0$ as $k \to \infty$, we can rearrange the above display so the left hand side is a telescoping series, yielding 
$$f \left( (1-\eta)T \right) \|u(T,\cdot)\|^2_{L^2(\Omega)} \leq \int_0^T \sup_{\omega_1 \times \omega_2}|u(t,\cdot)|^2 dt,$$
leading to \eqref{eq:ObsCartesianProducts} recalling the form of $f$.
\end{proof}

We will now use Proposition~\ref{prop:Obversability+spectral} to supply examples in higher dimensions. Combining \Cref{thm:spectralopendirichlet}, \Cref{thm:1D_Samb} and \Cref{prop:Obversability+spectral}, we exhibit observation sets of dimension $n-2+\delta$.

\begin{corollary}\label{cor:MultiD_lower_haussdorf_dimension}
Let $T>0$, $\delta>0$ and $\Omega=(0,1)\times \Omega'$ where $\Omega' \subset \R^{n-1}$. Let $x_0 \in (0,1)$ be a algebraic number of order $d>1$ and $\omega \subset \Omega'$ be a measurable subset satisfying $\mathcal C^{n-2+\delta}_{\mathcal H}(\omega)>0$. The heat equation \eqref{eq:Heat_equation} is observable from $\{x_0\}\times \omega$ at time $T$.
\end{corollary}

Finally, from \Cref{cor:spec-omega-alpha} and \Cref{prop:Obversability+spectral} we also obtain the following result. 

\begin{corollary}\label{cor:obs-omega-alpha}
    Let $\alpha,T>0$, $n \ge 2$, and $x_0$ an algebraic number of degree greater than one. Then the heat equation \eqref{eq:Heat_equation} on $\Omega=(0,1)^n$ is observable from $\{x_0\}\times \omega_\alpha \times \cdots \times \omega_\alpha$ at time $T$.
\end{corollary}

\subsection{Open questions}

By looking at \Cref{l:gamma-omega-alpha} and \Cref{cor:obs-omega-alpha}, the observation set $\{x_0\} \times \omega_{\alpha}$ has one accumulation point (which is the minimal number in a bounded domain), namely $(x_0,0)$,
and the set becomes even more concentrated there as $\alpha \to \infty$.  To this end we pose the following open question concerning the set
    \[ \omega_{\mathrm{exp}} = \left\{ 2^{-k}: k=1,2,\ldots \right\}.\]
It is relatively straightforward to check that $\gamma_k(\omega_{\mathrm{exp}}) \sim C^{-k^2}$, so that \eqref{e:gammak} fails and we do not obtain $\mathrm{spec}((0,1),\omega_{\mathrm{exp}},C,\beta)$ except when $\beta=1$. To conclude observability of the heat equation of course we need $\beta \in (0,1)$.
\begin{question}\label{q:exp}Is the heat equation \eqref{eq:Heat_equation} on $\Omega = (0,1)$ observable from $\omega_{\text{exp}}$? Second, is the heat equation \eqref{eq:Heat_equation} on $\Omega=(0,1)^2$ observable from
    \[ \{x_0\} \times \omega_{\mathrm{exp}}, \quad \text{or} \quad \omega_{\mathrm{exp}} \times \omega_{\mathrm{exp}} \]
for some point $x_0 \in (0,1)$?
\end{question}

These questions are seemingly independent since the observability inequality does not respect Cartesian products. However, if one could prove $\mathrm{spec}((0,1),\omega_{\mathrm{exp}},C,\beta)$ for some $\beta<1$, then all three questions would be answered in the affirmative. 

Second, we were unable to prove \Cref{prop:Obversability+spectral} in the full generality of \Cref{prop:spectral_cartesian_product2} using sections and projections of general non-product-type observation sets. We wonder if such a generalization is true, though evidence against it is given by the diagonal $D = \{ (x,x) : x \in (0,1) \} \subset (0,1)^2$. Being a nodal set of the eigenfunction $\phi_{n,m}(x,y) = \sin(\pi nx)\sin(\pi my)-\sin(\pi n y)\sin(\pi m x)$, $D$ cannot be an observation set. However, many of its slices are points which support the observability inequality by \Cref{thm:1D_Samb}, and if there are enough of such slices with uniform control, then the spectral inequality will hold on the projection of those points. 
Other than the negative results of $D$ and the positive results provided by \Cref{cor:MultiD_lower_haussdorf_dimension} we do not know of any other algebraic curves which are or are not observation sets. 
\begin{question}\label{q:line}
Is the heat equation on $(0,1)^2$ observable from any algebraic curve other than a line parallel to one of the axes?
\end{question}
A good starting point to tackle such a question seems to consider \cite{BR12} where the authors are able to derive observability inequality for the Laplace eigenfunctions on the flat two- or three-dimensional torus by an observation supported in a real analytic hypersurface with nonzero curvature.

\bigskip 

\appendix
\section{Approximate observability \label{app:approximate_observability}}

The next result seems to be well-known by specialists but we were not able to find a reference. The proof is inspired by \cite{CZ04}.
See Remark \ref{rmk:approximate_observability} for the definition of approximate observability. 
\begin{proposition}
Let $\Omega$ be a bounded open set in $\R^n$. 
The heat equation \eqref{eq:Heat_equation} is approximately observable from $\omega \subset \Omega$ in time $T>0$ if and only if $\omega$ is not included into the nodal set (zero set) of any Dirichlet Laplacian eigenfunction. 
\end{proposition}
Note that the condition is independent on time $T>0$. It holds for all or none $T>0$. 
\begin{proof}
If $\omega$ is included in the nodal set $\mathcal{N}(\varphi_\lambda)= \{x \in \Omega : \varphi_\lambda(x) = 0\}$ of a Dirichlet eigenfunction then the solution $u(t,x)=e^{-\lambda t} \varphi_\lambda(x)$ fails to satisfy the approximate observability condition from $\omega$.

Conversely, assume that no eigenfunction identically vanishes on $\omega$. Let $u$ be the solution to the heat equation \eqref{eq:Heat_equation} such that $u_{(0, T) \times \omega}=0$. Denote by $(\lambda_n)_{n \ge 1}$ the sequence of Dirichlet eigenvalues of the Laplacian and $(\mu_n)_{n \geq 1}$ their respective multiplicity. Let $(\varphi_{n,k})_{\substack{n \ge 1,\\ 1\le k \le \mu_n}}$ be an orthonormal basis associated. Then if $u$ is a solution of the heat equation, we have 
\[u(t,\cdot)= \sum_{n \ge 1} e^{-\lambda_n t} \sum_{k=1}^{\mu_n} a_{n,k} \varphi_{n,k},\]
with $\sum |a_{n,k}|^2 < \infty$.
Moreover, the Laplacian generates an analytic semigroup, so in particular 
$t \mapsto \|u(t, \cdot)\|_{L^2(\omega)}$ is analytic on the positive real half-line. Therefore the identity principle implies
\[u(t, \cdot)|_{\omega} = 0 \quad \text{for all } t >0.\]
Multiplying the previous equality by $e^{\lambda_1 t}$, we get
\[ a_{1,1}\varphi_{1,1}(x) + \sum_{n \ge 2} e^{(\lambda_1-\lambda_n) t} \sum_{k=1}^{\mu_n} a_{n,k} \varphi_{n,k}(x) = 0, \quad \text{for a.e } x \in \omega, \text{ for all } t>0. \]
Note that the first eigenvalue is simple. Finally the second term tends to zero as $t$ goes to infinity, so $a_{1,1}\varphi_{1,1}=0$ on $\omega$ and then $a_{1,1}=0$. A straightforward induction yields $a_{n,k}=0$ for all $n$ and $k$. Thus $u(T, \cdot)=0$. 
\end{proof}

\bibliographystyle{alpha}
\small{
\newcommand{\etalchar}[1]{$^{#1}$}

}

\end{document}